\documentclass{amsart}
\usepackage{amsmath,amssymb,tikz-cd,stmaryrd,tikz-cd,enumitem,mathtools,microtype,hyperref}

\newtheorem{thm}{Theorem}[section]
\newtheorem{introthm}{Theorem}

\newtheorem{lemma}[thm]{Lemma}
\newtheorem{prop}[thm]{Proposition}
\newtheorem{cor}[thm]{Corollary}

\theoremstyle{definition}
\newtheorem{ex}[thm]{Example}
\newtheorem{defi}[thm]{Definition}

\title[General matrix differential equations]{Simplifying matrix differential equations with general coefficients}
\author{Man Cheung Tsui}
\curraddr{Dept. of Mathematics\\
Florida State University\\ 
Tallahassee, FL 32304\\
USA}
\email{mktsui@fsu.edu}

\DeclareMathOperator{\ed}{ed}
\DeclareMathOperator{\res}{res}
\newcommand{\ded}{\ed^{\delta}}

\DeclareMathOperator{\Frac}{Frac}
\DeclareMathOperator{\SO}{SO}
\DeclareMathOperator{\OO}{O}
\DeclareMathOperator{\PGL}{PGL}
\DeclareMathOperator{\SL}{SL}
\DeclareMathOperator{\Gal}{Gal}
\newcommand{\Dgal}{\Gal^{\delta}}
\DeclareMathOperator{\trdeg}{trdeg}
\DeclareMathOperator{\GL}{GL}
\newcommand{\dtrdeg}{\trdeg^{\delta}}

\DeclareMathOperator{\diag}{diag}

\newcommand{\anglebrac}[1]{\langle{#1}\rangle}

\begin{document}
\maketitle

\begin{abstract}
We show that the $n\times n$ matrix differential equation $\delta(Y)=AY$ with $n^2$ general coefficients cannot be simplified to an equation in less than $n$ parameters by using gauge transformations whose coefficients are rational functions in the matrix entries of $A$ and their derivatives. Our proof uses differential Galois theory and a differential analogue of essential dimension. We also bound the minimum number of parameters needed to describe some generic Picard-Vessiot extensions.
\end{abstract}

\section{Introduction}
\label{section:introduction}
A classical problem in mathematics is simplifying polynomials. Let $F$ be a field of characteristic zero. The general quadratic polynomial $x^2+ax+b$ in two algebraically independent coefficients $a$ and $b$ over $F$ simplifies to a polynomial $y^2+c$ in one coefficient $c$ by taking $y = x+a/2$ and $c = b - a^2/4$. Similarly the general cubic $x^3+ax^2+bx+c$ simplifies to a polynomial of the form $y^3 + dy+d$ by using translations and dilations. More generally, one can simplify the general polynomial of degree $n$ by means of a larger class of transformations called Tschirnhaus transformations. Let $\tau(n)$ denote the number of algebraically independent coefficients needed to write down such a polynomial after simplifying by Tschirnhaus transformations. We just saw $\tau(2) = \tau(3) = 1$. The works of classical mathematicians, J. Buhler and Z. Reichstein in \cite{buhler1997essential}, and subsequent authors give bounds on $\tau(n)$. However, the exact values of $\tau(n)$ remain unknown for $n\ge 8$.

This paper studies the differential equations analogue of the above problem. Namely, we quantify precisely how much the general matrix differential equation and the general homogeneous linear differential equation simplify. 

From now on, $(F,\delta)$ is a \emph{differential field}, i.e., $F$ is a field equipped with a derivation map $\delta:F\rightarrow F$. For example, the complex function field $\mathbb C(x)$ with the derivation $d/dx$ is a differential field. We will always assume that $F$ has characteristic zero and that all differential fields discussed in this paper contain $F$.

Consider the \emph{general $n\times n$ matrix differential equation} $\delta(Y)=AY$ where the $n^2$ matrix entries $a_{ij}$ of $A$ are \emph{differentially algebraically independent} over $F$, i.e., $a_{ij}$ and all their formal derivatives $a_{ij}^{(k)}$ are algebraically independent over $F$.

We wish to simplify this equation by means of gauge transformations. Given two $n\times n$ matrix differential equations $\delta(U)=CU$ and $\delta(V)=DV$ over a differential field $L$, we say that $\delta(V)=DV$ is \emph{gauge equivalent} to $\delta(U)=CU$ if there exists a matrix $P\in \GL_{n}(L)$ such that $V=PU$, in which case we have
$$
D = PCP^{-1}+\delta(P)P^{-1}.
$$
The transformation from $\delta(U) = CU$ to $\delta (V) = DV$ is called a \emph{gauge transformation}. By the cyclic vector theorem \cite[Proposition 2.9]{van2012galois}, the general matrix differential equation $\delta(Y)=AY$ is gauge equivalent to an equation of the form 
\begin{equation}\label{eq:gauge}
    \delta(Z)=BZ,\quad B = \begin{pmatrix}
0 & 1 & 0 & \cdots & 0\\
0 & 0 & 1 & \cdots & 0\\
\vdots &\vdots & \vdots & \ddots & \vdots\\
0 & 0 & 0 & \cdots & 1\\
-b_0 & -b_1 & -b_2 & \cdots & -b_{n-1}\\
\end{pmatrix},
\end{equation}
over the differential field 
$$
K=F\left(a_{ij}^{(k)}\;\middle|\; i,j\in\{1,2,...,n\}, k\ge 0\right).
$$
\sloppy Thus $\delta(Y)=AY$ simplifies to a differential equation in the $n$ ``parameters'' $b_{0},...,b_{n-1}\in K$.

In general, let $\gamma(n)$ be the minimum number of parameters we may simplify $\delta(Y)=AY$ to by using gauge transformations over $K$. More precisely, we define $\gamma(n)$ to be the minimum over the number of differentially algebraically independent entries of $B$, ranging over all equations $\delta(Z)=BZ$ gauge equivalent to $\delta(Y)=AY$ over $K$. We just saw that $\gamma(n)\le n$. The main result of this paper is the following.

\begin{introthm}[Theorem \ref{thm:gn=n}]\label{introthm:gn=n}
We have $\gamma(n) = n$ for all $n\ge 1$.
\end{introthm}

The matrix differential equation \eqref{eq:gauge} is associated to the homogeneous linear differential equation
\begin{equation}\label{eq:homog-lin}
y^{(n)}+b_{n-1}y^{(n-1)}+\cdots+b_{1}y^{(1)}+b_0y = 0
\end{equation}
(see \cite[page 8]{van2012galois}). Theorem \ref{introthm:gn=n} tells us the coefficients $b_{0},...,b_{n-1}$ are differentially algebraically independent over $F$. In particular, we can view \eqref{eq:homog-lin} as the \emph{general homogeneous linear differential equation}. Moreover, Theorem \ref{introthm:gn=n} says that we cannot simplify \eqref{eq:homog-lin} further using gauge transformations over $K$. In this sense, homogeneous linear differential equations are the most economical way to write matrix differential equations up to gauge transformations.
This is in direct contrast to the case of polynomials --- the transformation $y = x+a_{n-1}/n$ reduces the polynomial
$$
x^{n}+a_{n-1}x^{n-1}+\cdots+a_1x+a_0
$$
to a polynomial with no $y^{n-1}$ term. 

However, we can simplify \eqref{eq:homog-lin} with an additional operation. Let $c$ be an element of a differential field containing $K$ that satisfies $\delta(c) = \frac{1}{n}b_{n-1}c$, or informally $c = e^{\frac{1}{n}\int b_{n-1}}$. The \emph{exponential product} transformation $z = cy$ simplifies \eqref{eq:homog-lin} to an equation of the form
$$
z^{(n)}+0\cdot z^{(n-1)}+c_{n-2}z^{(n-2)}+\cdots+c_{1}z^{(1)}+c_0z = 0
$$
in the $n-1$ parameters $c_0, ...,c_{n-2}$. Thus the next natural problem (outside the scope of this paper) is to determine how much \eqref{eq:homog-lin} simplifies by using gauge transformations and exponential products. Since the transformation $z=cy$ corresponds to a gauge transformation of \eqref{eq:gauge} over the differential field $K(c)$, the problem can be phrased more precisely in the following way. Let $L/K$ be an \emph{extension by exponentials}, i.e., let there be containments of differential fields $K = K_0 \subseteq K_1 \subseteq \cdots \subseteq K_m = L$ that satisfy $K_{i+1} = K_{i}(e^{\int a_{i}})$ for some $a_{i}\in K_{i}$. The rephrased problem is to determine how much \eqref{eq:gauge} simplifies by gauge transformations over such extensions $L$.\bigskip

The proof of Theorem \ref{introthm:gn=n} relies on adapting the notion of essential dimension to differential algebra. The theory of essential dimension, introduced by J. Buhler and Z. Reichstein in \cite{buhler1997essential}, bounds the number $\tau(n)$. This theory has since been greatly generalized and found many applications. For example, it is used to bound the minimum number of parameters of generic Galois extensions (see \cite[Section 8.2]{jensen2002generic}). The introductory article \cite{reichstein2012whatis} and the survey papers \cite{merkurjev2017essential} and \cite{reichstein2010essential} explain the theory of essential dimension and its applications.

In differential algebra, there is an analogous Galois theory for matrix differential equations. The resulting differential Galois extensions are called \emph{Picard-Vessiot extensions}, and many researchers 
have constructed \emph{generic Picard-Vessiot extensions} --- extensions which parametrize all other Picard-Vessiot extensions with the same differential Galois group. See for instance the papers \cite{goldman1957specialization}, \cite{juan2016generic},  \cite{juan2007generic}, and \cite{seiss2020general} for more on generic Picard-Vessiot extensions and related concepts. Oftentimes the constructed generic Picard-Vessiot extension requires many parameters, and it is not clear if there exists a generic Picard-Vessiot extension with the same differential Galois group that has fewer parameters. Corollary \ref{cor:generic-PV} gives lower bounds on the number of parameters required by generic Picard-Vessiot extensions with differential Galois groups $\GL_n$,  $\mathbb{G}^n_m$, $\mathbb{G}^n_a$, $(\mathbb{Z}/k\mathbb{Z})^n$, $\OO_n$, $\SL_n$, $\mathbb{U}_n$, and $\SO_n$.

\textbf{Acknowledgements:} This work contains results from the author's thesis. The author thanks Anand Pillay and R\'{e}mi Jaoui for spotting a problem in an early version of Proposition \ref{prop:ded-gm-ga}, David Harbater for his generous input and helpful comments especially towards Proposition \ref{prop:ded-gm-ga}, the Kolchin Seminar in Differential Algebra for the opportunity to present an early version of this work, Michael Wibmer for helpful conversations, Mark van Hoeij and the referee for helpful comments, and Julia Hartmann for her patient advising and ample support.

\section{Background}\label{section:background}

We first review some notions in differential algebra. For details, see \cite[Chapter 1]{van2012galois} and \cite[Chapters 1 -- 2]{kolchin1973differential}.

Let $\mathbb N = \{0,1,2,...\}$. For us, a \emph{ring} is unital and commutative. 

Let $R$ be a \emph{differential ring}, i.e., $R$ is a ring equipped with a derivation $\delta:R\rightarrow R$. We often write $r^{(n)}\coloneqq(\delta\circ\cdots\circ \delta)(r)$ for the $n$-fold application of the derivation $\delta$ on $r$. We will not use the notation $r'$ to mean $\delta(r)$. For a subset $\pmb{r}$ of $R$, we write $\pmb{r}^{(n)}$ for the set $\left\lbrace r^{(n)}\;\middle|\; r\in\pmb{r}\right\rbrace$. The \emph{ring of constants} of $R$ is the subring $C_{R} \coloneqq \left\lbrace r\in R\;\middle|\;\delta(r) = 0\right\rbrace$
of $R$. If $C_{R}$ is a field, which occurs when $R$ is a differential field, we say $C_R$ is the  \emph{field of constants} of $R$. A \emph{differential homomorphism} is a ring homomorphism of differential rings that commutes with the derivations. Given an algebra $C'$ over the field of constants $C$ of $F$, the ring $F\otimes_C C'$ becomes a differential ring with derivation defined by $\delta(r\otimes c)\coloneqq\delta(r)\otimes c$.

\emph{Differential algebras} and \emph{differential field extensions} are similarly defined as algebras and field extensions with compatible derivations and whose structure ring homomorphisms are differential homomorphisms.

\subsection{Differential polynomials and differential transcendence degree}
Let $I$ be a set and let $x_{i}^{(j)}$ be algebraic indeterminates over $F$ for $i\in I$ and $j\in \mathbb N$. A \emph{differential polynomial ring} over $F$ in the indeterminates $x_{i}$ for $i\in I$ is the differential $F$-algebra 
$$F\left\lbrace x_i\;\middle| \; i\in I\right\rbrace\coloneqq F\left[x_{i}^{(j)}\;\middle|\; i\in I, j\in \mathbb{N}\right]$$
whose derivation $\delta$ is determined by $\delta x_{i}^{(j)}=x_{i}^{(j+1)}$ for all $i\in I$ and $j\in \mathbb N$. Such $x_i$ are called \emph{differential indeterminates} over $F$.
Given a differential $F$-algebra $R$ and a subset $\{r_{i}\mid i\in I\}$ of $R$, the image of the differential homomorphism 
$$\varphi:F\left\lbrace x_i\;\middle| \; i\in I\right\rbrace\rightarrow R:\quad x_{i}\mapsto r_{i}$$
is denoted by $F\left\lbrace r_i\;\middle| \; i\in I\right\rbrace$. If $\varphi$ is injective, we say that $\{r_{i}\mid i\in I\}$ is \emph{differentially algebraically independent} over $F$. Otherwise $\{r_{i}\mid i\in I\}$ is \emph{differentially algebraically dependent} over $F$. The field of fractions of $F\left\lbrace r_i\;\middle| \; i\in I\right\rbrace$, denoted by $F\left\langle r_i\;\middle| \; i\in I\right\rangle$, is a differential field with derivation defined by the quotient rule. Given differential indeterminates $x_i$ over $F$, the elements of $F\left\lbrace x_i\;\middle| \; i\in I\right\rbrace$ and $F\left\langle x_i\;\middle| \; i\in I\right\rangle$ are called \emph{differential polynomials} and \emph{differential rational functions} over $F$, respectively.

Let $K/F$ be a differential field extension. A subset $\pmb{r}\subseteq K$ \emph{generates} $K/F$ if $K=F\anglebrac{\pmb{r}}$, and $K/F$ is \emph{finitely generated} if $K=F\anglebrac{\pmb{r}}$ for some finite set $\pmb{r}$. A \emph{differential transcendence basis} of $K/F$ is a differentially algebraically independent subset of $K$ over $F$ that is maximal with respect to inclusion. We say that $K/F$ is \emph{purely differentially transcendental} if $K$ contains a differential transcendence basis over $F$ that generates the extension. By \cite[Section 2.9, Theorem 4(c)]{kolchin1973differential}, any two differential transcendence bases of $K/F$ have the same cardinality. The \emph{differential transcendence degree} of $K/F$, denoted by $\trdeg^\delta_F K$, is the cardinality of any given differential transcendence basis of $K/F$.

In this language, $\gamma(n) = \min \dtrdeg_F K_0$ where the minimum is taken over those differential fields $K_0$ that contain both $F$ and the coefficients of some equation gauge equivalent to the general $n\times n$ matrix differential equation.

The following result on finitely generated extensions is needed later.

\begin{prop}[\cite{kolchin1973differential}, Section 2.11, Proposition 14, characteristic $p = 0$ case]\label{prop:fg}

Let~$L/M/K$ be differential field extensions. If $L/K$ is a finitely generated differential field extension, then $M/K$ is a finitely generated differential field extension.
\end{prop}

\subsection{Picard-Vessiot extensions}\label{subsection:pv}
In this subsection, we assume the field of constants $C$ of $F$ is algebraically closed. An \emph{algebraic variety} is a variety over $C$ in the classical sense as in \cite{springer1994linear}. A \emph{linear algebraic group} is an affine algebraic group.

Let $\delta(Y)=AY$ be an $n\times n$ matrix differential equation over $F$. 
A \emph{Picard-Vessiot field} for $\delta(Y)=AY$ over $F$ is a differential field $K$ containing $F$ such that $C_K = C$ and such that $K = F\anglebrac{z_{ij}\mid 1\le i, j\le n}$ for some matrix $Z = (z_{ij})\in \GL_{n}(K)$ satisfying $\delta(Z)=AZ$. Such a matrix $Z$ is called a \emph{fundamental solution matrix} of the equation. By \cite[Proposition 1.20]{van2012galois}, a Picard-Vessiot field for $\delta(Y) = AY$ over $F$ exists and is unique up to differential isomorphism.

Let $K$ be a Picard-Vessiot field over $F$ for an $n\times n$ matrix differential equation $\delta(Y)=AY$ with fundamental solution matrix $Z$. The \emph{differential Galois group} of $K/F$ is the group $\Dgal(K/F)$ of differential $F$-algebra automorphisms of $K$. Let $G=\Dgal(K/F)$. By \cite[Observations 1.26(2)]{van2012galois}, $\sigma\in G$ acts on $Z$ entrywise and there exists a matrix $C(\sigma)\in\GL_n(C)$ such that $\sigma(Z) = Z\, C(\sigma)$. This gives an injective group homomorphism $\varphi:G\rightarrow \GL_n(C)$ via $\sigma\mapsto C(\sigma)$. By \cite[Theorem 1.27]{van2012galois}, $G$ has the structure of a linear algebraic group and the map $\varphi$ extends to a homomorphism of algebraic groups $G\to \GL_n$ over $C$.

For a linear algebraic group $G$ over $C$, a \emph{$G$-Picard-Vessiot extension} is a Picard-Vessiot field $K$ over $F$ equipped with an isomorphism of algebraic groups $G\to \Dgal(K/F)$, and we view $G$ as acting on $K$ via $\Dgal(K/F)$. Such an extension satisfies the equality $\trdeg_F K = \dim(G)$ by \cite[Corollary 1.30(3)]{van2012galois}. If $G$ is finite, the $G$-Picard-Vessiot extensions over $F$ are precisely the $G$-Galois extensions over $F$ by \cite[Exercise 1.24]{van2012galois} and \cite[Proposition 1.34(3)]{van2012galois}.

Let $K/F$ be a $G$-Picard-Vessiot extension. By \cite[Proposition 1.34]{van2012galois}, there is a Galois correspondence between the intermediate differential fields $K/M/F$ and the closed subgroups $H$ of $G$, given by taking $M$ to $\Dgal(K/M)$ and $H$ to $K^H$. If $H$ is a normal closed subgroup of $G$, then $K^H/F$ is a $(G/H)$-Picard-Vessiot extension. If $G^{\circ}$ is the connected component of $G$, then $K^{G^{\circ}}/F$ is both a finite Galois extension with Galois group $G/G^{\circ}$ and the algebraic closure of $F$ in $K$.

\begin{ex}\label{ex:general-de}
Let $L$ be a Picard-Vessiot field over $K = F\langle a_{ij}\rangle$ for the general matrix differential equation $\delta(Y)=AY$. This $L$ is generated by the $n^2$ entries of a fundamental solution matrix $Y_0$ for $\delta(Y)=AY$. Since $\dtrdeg_FL = \dtrdeg_FK = n^2$, the entries of $Y_0$ are differentially algebraically independent over $F$. Thus right matrix multiplication of $Y_0$ by elements $\sigma\in \GL_n(C)$ defines differential automorphisms $\sigma$ of $L$ over $F$. The computation
$$A.\sigma = (\delta(Y_0)\cdot Y_0^{-1}).\sigma = (\delta(Y_0).\sigma) \cdot (Y_0.\sigma)^{-1} = \delta(Y_0)\cdot \sigma\cdot \sigma^{-1}\cdot Y_0^{-1} = A$$
of applying $\sigma$ entrywise to $A$ shows that $\sigma$ fixes the entries of $A$ and hence $K$. So $\Dgal(L/K) = \GL_n(C)$.
\end{ex}

\section{Differential essential dimension}\label{section:ded}

This section defines the differential essential dimension of Picard-Vessiot extensions. We also compute this invariant for certain $G$-Picard-Vessiot extensions when $G = \mathbb G_{m}^{n}$, $\mathbb G_{a}^{n}$, and $(\mathbb Z/k \mathbb Z)^{n}$ for $k>0$ (see Corollary \ref{cor:ded-gm-ga}). 

In this and subsequent sections, we assume that the field of constants $C$ of the differential field $F$ is algebraically closed, and that any differential field in discussion contains $F$ and has field of constant $C$.

We first explain why we may assume $C$ to be algebraically closed when proving $\gamma(n)\ge n$. Suppose that the general $n\times n$ matrix differential equation $\delta(Y)=AY$ over the differential field
$$K\coloneqq F\anglebrac{a_{ij}\mid 1\le i,j\le n}$$
is gauge equivalent to some equation $\delta(Z)=BZ$ over a differential subfield $K_0$ containing $F$. By \cite[Proposition 2.3]{bachmayr2021large}, both $F' = \Frac(F\otimes_C \overline{C})$ and $K'_0=\Frac(K_0\otimes_C\overline{C})$ are differential fields with field of constants $\overline{C}$ (here $\overline{C}$ denotes the algebraic closure of $C$). Since $\dtrdeg_F K_0 \ge \dtrdeg_{F'} K_0'$, it suffices to show $\dtrdeg_{F'} K'_0\ge n$ to establish $\gamma(n)\ge n$. Thus we may assume $C$ to be algebraically closed. 

We now explain how to describe $\gamma(n)$ in a new language. By Example \ref{ex:general-de}, the general matrix differential equation $\delta(Y)=AY$ determines a $\GL_n$-Picard-Vessiot extension $L/K$. Suppose that an equation $\delta(Z)=BZ$  gauge equivalent to $\delta(Y)=AY$ has coefficients in a differential subfield $K_0$ of $K$ containing $F$. Then $\delta(Z)=BZ$ determines a $\GL_n$-Picard-Vessiot extension $L_0/K_0$ with $L_0\subseteq L$ such that the compositum $L_0K$ equals $L$. This motivates the following definition.

\begin{defi}\label{def:ded}
Suppose that $L/K$ and $L_0/K_0$ are $G$-Picard-Vessiot extensions such that $L_0\subseteq L$ and $K_0\subseteq K$. If $L_0K = L$, the extension $L/K$ \emph{descends} to $L_0/K_0$, and $K_0$ is a \emph{differential field of definition} of $L/K$.

The \emph{differential essential dimension} of a $G$-Picard-Vessiot extension $L/K$ is the smallest differential transcendence degree of a differential field of definition of $L/K$. We denote this number by $\ded_F(L/K)$.
\end{defi}

The preceding discussion immediately gives

\begin{prop}\label{prop:gamma=GL}
Let $L/K$ be the Picard-Vessiot extension for the general $n\times n$ matrix differential equation. Then $\gamma(n)\ge \ded_F(L/K)$.
\end{prop}

Let $L/K$ and $L_{0}/K_{0}$ be Picard-Vessiot extensions with $L_0 \subseteq L$ and $K_{0} \subseteq K$. We define the restriction map 
$$\res:\Dgal(L/K)\to\Dgal(L_0/K_0)$$ to
take a differential automorphism $\sigma$ to its restriction $\sigma\vert_{L_0}$. 

This restriction map is a homomorphism of algebraic groups. To see this, let $Y$ and $Z$ be fundamental solution matrices for some $m\times m$ and $n\times n$ matrix differential equations associated to $L/K$ and $L_0/K_0$, respectively. Suppose that $\sigma\in \Dgal(L/K)$ and $\tau\in \Dgal(L_0/K_0)$ act on $Y$ and $Z$ entrywise by $\sigma(Y) = Y\, C(\sigma)$ and $\tau(Z) = Z\, D(\tau)$. Then $\Dgal(L/K)$ embeds into $\GL_{m+n}(C)$ taking $\sigma$ to $\diag(C(\sigma),D(\res(\sigma)))$, and $\Dgal(L_0/K_0)$ embeds into $\GL_{n}(C)$ taking $\tau$ to $D(\tau)$. Under this identification, the restriction map is the polynomial map that takes the block diagonal matrix $\diag(C(\sigma),D(\res(\sigma)))$ to $D(\res(\sigma))$.

We can express descent in other ways.

\begin{lemma}\label{lem:descent-tfae}
Let $L/K$ and $L_0/K_0$ be $G$-Picard-Vessiot extensions with $L_0\subseteq L$ and $K_0\subseteq K$. The following are equivalent:
\begin{enumerate}
    \item\label{descent-1} $L/K$ descends to $L_0/K_0$.
    \item\label{descent-2} The restriction map $\res:\Dgal(L/K)\to\Dgal(L_0/K_0)$
    is injective.
    \item\label{descent-3} The restriction map is bijective.
    \item\label{descent-4} There exists a $G$-action on $L_0$ such that the inclusion $L_0\subseteq L$ is $G$-equivariant.
\end{enumerate}
\end{lemma}

\begin{proof}
(\ref{descent-1}) $\Leftrightarrow$ (\ref{descent-2}). The kernel $\ker(\res) = \Dgal(L/L_{0}K)$ is trivial exactly when $L = L_0K$.

(\ref{descent-3}) $\Rightarrow$ (\ref{descent-2}) is immediate.

(\ref{descent-2}) $\Rightarrow$ (\ref{descent-3}). Identify $G$ with $\Dgal(L_0/K_0)$. By \cite[Proposition 2.2.5(ii)]{springer1994linear}, the image $G'$ of the restriction map is a closed subgroup of $G$. By assumption (\ref{descent-2}), $G$ and $G'$ are isomorphic and have the same dimension as varieties. The algebraic group $G$ is then a finite disjoint union of left cosets of $G'$. But $G$ and $G'$ have the same number of connected components, so $G = G'$.

(\ref{descent-3}) $\Rightarrow$ (\ref{descent-4}). Let $G$ act on $L_0$ via the composition of the map $G\to \Dgal(L/K)$ defining the $G$-action on $L$ and the restriction map.

(\ref{descent-4}) $\Rightarrow$ (\ref{descent-3}). The $G$-equivariance of the inclusion $L_0 \subseteq L$ gives the commutative diagram
$$
\begin{tikzcd}[row sep=1em,column sep=0.5em]
& G\arrow[dl]\arrow[dr]\\
\Dgal(L/K)\arrow[rr,"\res"]&& \Dgal(L_0/K_0).
\end{tikzcd}
$$
Here the diagonal maps are the isomorphisms defining the $G$-actions on $L$ and $L_0$, so the restriction map is a bijection.
\end{proof}

We now bound the differential essential dimension of Picard-Vessiot subextensions.

\begin{cor}\label{cor:descent-subPV}
Suppose that a $G$-Picard-Vessiot extension $L/K$ descends to an extension $L_0/K_0$. 
\begin{enumerate}[label={(\arabic*)}]
    \item\label{descent-ded-1} If $H$ is a closed subgroup of $G$, the $H$-Picard-Vessiot subextension $L/L^H$ descends to $L_0/L_0^H$. In particular, $\ded_F(L/L^H)\le \ded_{F}(L/K)$.
    \item\label{descent-ded-2} If $H$ is a closed normal subgroup of $G$, the $(G/H)$-Picard-Vessiot subextension $L^H/K$ descends to $L_0^H/K_{0}$. In particular, $\ded_F(L^H/K)\le \ded_{F}(L/K)$.
\end{enumerate}
\end{cor}
\begin{proof}
\ref{descent-ded-1}. By Lemma \ref{lem:descent-tfae}, we may assume the inclusion $L_0 \subseteq L$ is $G$-equivariant and thus $H$-equivariant. Therefore $L/L^H$ descends to $L_0/L_0^H$ and
$$\dtrdeg_F K_0 = \dtrdeg_{F}L_0^H\ge\ded_F(L/L^H).$$
Taking the infimum over all differential fields of definition $K_0$ of $L/K$ gives $\ded_F(L/K)\ge \ded_F(L/L^H)$.

\ref{descent-ded-2}. If $H$ is normal in $G$, the above $G$-equivariant inclusion $L_0 \subseteq L$ restricts to a $(G/H)$-equivariant inclusion $L_0^H \subseteq L^H$ so $L^H/K$ descends to $L_0^H/K_0$. We likewise deduce  $\ded_{F}(L/K)\ge \ded_F(L^H/K)$.
\end{proof}

We next find the differential essential dimensions of some extensions.

\begin{prop}\label{prop:ded-gm-ga}
Let $G$ be a linear algebraic group over $C$ and let $u(Y)$ be a differential rational function over $C$. Suppose the following three conditions are satisfied for any differential field $E$ over $F$ with field of constants $C$:
\begin{enumerate}[label={(\arabic*)}]
    \item\label{cond:1} For $n\ge 0$, any $G^n$-Picard-Vessiot extension over $E$ has the form $E\anglebrac{ y_1,...,y_n}$ with $u(y_i)\in E$ for $i=1,...,n$.
    \item\label{cond:2} Suppose that $E\anglebrac{y}$ and $E\anglebrac{z}$ give the same $G$-Picard-Vessiot extension over $E$ and satisfy $u(y),u(z)\in E$. Then there exists a differential polynomial $v(Y,Z)$ over $F$ such that $v(y,z)$ is a nonzero element in $E$. Moreover there exists an automorphism $\phi$ of $E\anglebrac{z}$ for which $u(\phi(T)) = u(T)$ and $v(\phi(T),z)\neq v(T,z)$.
    \item\label{cond:3} For $n\ge 0$ and differential indeterminates $y_{1},...,y_{n}$ over $F$, the extension $L=F\anglebrac{y_{1},...,y_{n}}$ over $K=F\anglebrac{ u(y_{1}),...,u(y_{n})}$ is a $G^n$-Picard-Vessiot extension.
\end{enumerate}
Then the extension $L/K$ in condition \ref{cond:3} satisfies $\ded_F(L/K) = n$ for all $n\ge 0$.
\end{prop}

\begin{cor}\label{cor:ded-gm-ga}
Let $y_1,...,y_n$ be differential indeterminates over $F$ and let $L = F\anglebrac{ y_1,...,y_n}$. Then $\ded_F(L/L^{G^n}) \ge n$ for $G = \mathbb G_{m}$, $\mathbb G_{a}$, and $\mathbb Z/k\mathbb{Z}$ (for $k\ge 2$).
\end{cor}
\begin{proof}
To prove this corollary, it suffices to check the conditions stated in Proposition \ref{prop:ded-gm-ga} hold for $G = \mathbb G_{m}$, $\mathbb G_{a}$, and $\mathbb Z/k\mathbb{Z}$.

Let $G = \mathbb Z/k\mathbb{Z}$. Since $G^n$ is finite, $G^n$-Picard-Vessiot extensions are the same as $G^n$-Galois extensions as recalled in Subsection \ref{subsection:pv}. Conditions \ref{cond:1} and \ref{cond:3} are verified by Kummer theory, taking $u(T)=T^k$. If $E(y)$ and $E(z)$ give the same $G$-Galois extension over a field $E$ and satisfy $y^k, z^k\in E$, there exists an integer $s$ coprime to $k$ such that $yz^s\in E$ by \cite[Section 14.7, Exercise 7(b), page 636]{dummit2004abstract}. Condition \ref{cond:2} is verified by taking $v(Y,Z) = YZ^s$ and $\phi(T)=\xi T$, where $\xi\neq 1$ is a $k$th root of unity.

Let $G = \mathbb G_{m}$. Conditions \ref{cond:1} and \ref{cond:3} are verified by the structure theorem for $\mathbb G_{m}^{n}$-Picard-Vessiot extensions \cite[Exercise 1.41]{van2012galois}, taking $u(T) = \delta(T)/T$. If $E\anglebrac{ y}$ and $E\anglebrac{ z}$ give the same $G$-Picard-Vessiot extension over a differential field $E$ and satisfy $\delta(y)/y, \delta(z)/z\in E$, there exist nonzero integers $r,s$ such that $y^rz^s\in E$ by the Kolchin-Ostrowski theorem \cite[pages 1155--1156]{kolchin1968algebraic}. Condition \ref{cond:2} is verified by taking $v(Y,Z) = Y^r Z^s$ and $\phi(T) = 2T$.

The results given in \cite[Exercise 1.41]{van2012galois} and \cite[pages 1155--1156]{kolchin1968algebraic} also apply to the case $G= \mathbb G_{a}$, taking $u(T) = \delta(T)$, $v(Y,Z) = rY + sZ$ for $r,s\in C^{\times}$ such that $v(y,z)\in F$, and $\phi(T)=T+1$.
\end{proof}

We now prove Proposition \ref{prop:ded-gm-ga}.

\begin{proof}[Proof of Proposition \ref{prop:ded-gm-ga}]
For the sake of contradiction, suppose the proposition is false. Consider the smallest $n$ for which it fails. The case of $n=0$ trivially holds so $n\ge 1$. The extension $L/K$ descends to some extension $L_{0}/K_{0}$ for some differential subfield $K_{0}$ of $K$ satisfying $\dtrdeg_{F}K_{0}<n$. 

For $i=1,...,n$, the $G$-Picard-Vessiot subextension $K\anglebrac{ y_i}/K$ of $L/K$ descends to some $G$-Picard-Vessiot subextension of $L_0/K_{0}$, which by condition \ref{cond:1}, necessarily has the form $K_{0}\anglebrac{ z_i}/K_0$ for some $z_{i}\in L_{0}$ satisfying $u(z_{i})\in K_{0}$. Thus we can write $L_0$ as $F\anglebrac{ z_1,...,z_n}$ and $K_0$ as $F\anglebrac{ u(z_1),...,u(z_n)}$.

$$
\begin{tikzcd}[row sep=small, column sep = small,/tikz/column 1/.append style={anchor=base east},/tikz/column 4/.append style={anchor=base west}]
F\anglebrac{ y_1,...,y_{n}}=&L\arrow[d,dash]\arrow[dr,dash]& &\\
F\anglebrac{ u(y_1),...,u(y_n)}=&K\arrow[dr,dash] & L_0\arrow[d,dash]&=F\anglebrac{ z_1,...,z_{n}}\\
&  & K_0&=F\anglebrac{ u(z_1),...,u(z_n)}
\end{tikzcd}
$$

Let $L' = F\anglebrac{ y_1,...,y_{n-1}}$, $K'=F\anglebrac{ u(y_1),...,u(y_{n-1})}$, $L_{0}' = F\anglebrac{ z_1,...,z_{n-1}}$, $K_{0}' = F\anglebrac{ u(z_1),...,u(z_{n-1})}$, and $L''=L'\anglebrac{ z_{n}}$. Note that the $G^{n-1}$-Picard-Vessiot extension $L'/K'$ descends to $L_0'/K_0'$. 

$$
\begin{tikzcd}[row sep=small, column sep = small,/tikz/column 4/.append style={anchor=base west},/tikz/column 1/.append style={anchor=base east}]
L'\anglebrac{ z_n}=&L''\arrow[d,dash]\arrow[dr,dash]& &\\
F\anglebrac{ y_1,...,y_{n-1}}=&L'\arrow[d,dash]\arrow[dr,dash]& L_0\arrow[d,dash] & = L'_0\anglebrac{ z_n}\\
F\anglebrac{ u(y_1),...,u(y_{n-1})}=&K'\arrow[dr,dash] & L_0'\arrow[d,dash]&=F\anglebrac{ z_1,...,z_{n-1}} \\
&  & K_0'&=F\anglebrac{ u(z_1),...,u(z_{n-1})}
\end{tikzcd}
$$

Since $n$ is the minimal value for which the proposition fails, we have $\dtrdeg_{F}L_0'=\dtrdeg_{F}K_0' = n-1$. Also we have $L_0 = L_0'\anglebrac{ z_n}$, and so the inequalities $$
n-1\ge \dtrdeg_{F}K_{0}=\dtrdeg_{F}L_{0}\ge\dtrdeg_{F} L_{0}' = n-1$$
now force $z_{n}$ to be differentially algebraic over $L_0'$ and hence over $L'$. Therefore $\dtrdeg_{F}L'' = \dtrdeg_{F}L' = n-1$.

To finish proving the proposition, it suffices to show that $y_n$ is differentially algebraic over $L''$, since this would then imply 
$$n = \dtrdeg_{F} L = \dtrdeg_{F} L''=n-1,$$ resulting in the desired contradiction. Recall that the extension $K\anglebrac{ y_{n}}/K$ descends to $K_{0}\anglebrac{ z_{n}}/K_{0}$.  By condition \ref{cond:2}, there exists a differential polynomial $v(Y,Z)$ over $F$ such that $d\coloneqq v(y_{n},z_{n})$ is a nonzero element in $K$. Since $K = F\anglebrac{ u(y_1),...,u(y_n)}=K'\anglebrac{ u(y_n)}$, we may view $d$ as $f(u(y_n))$ where $f$ is a differential rational function in one variable $T$ over $K'$. Therefore $v(y_n,z_n) = d = f(u(y_{n}))$. Since $d$ is nonzero, $f(T)$ is nonzero. Furthermore by condition \ref{cond:2}, there exists an automorphism $\phi$ that does not fix the differential rational function $g(T)\coloneqq f(u(T)) - v(T,z_n)$. Therefore $y_{n}$ satisfies a nontrivial relation $g(y_n)=0$ over $L''$ and so $y_n$ is differentially algebraic over $L''$.
\end{proof}

\section{Proof of the main theorem}\label{section:main-thm}

In this section, we prove our main result, Theorem \ref{thm:gn=n}, following the approach in \cite[Chapter 8]{jensen2002generic}. As before, we assume the field of constants $C$ of $F$ is algebraically closed, and all differential fields in discussion contain $F$ and have field of constants $C$. We first establish an analogue of the following.

\begin{prop}[Roquette]\label{prop:roquette}
Let $L/F$ be a field extension of finite transcendence degree, and let $M/F$ be a field subextension of $L(\pmb{x})/L$ for some set of algebraic indeterminates $\pmb{x}$ over $L$. Assume that $\trdeg_{F}M \le \trdeg_{F} L$. Then there exists a nonzero polynomial $h(\pmb{x})$ in $L[\pmb{x}]$ such that the following holds: for any tuple $\pmb{c}=(c_x\mid x\in\pmb{x})$ of $L$ satisfying $h(\pmb{c})\neq 0$, the $L$-algebra homomorphism $$L[\pmb{x}]_{\mathfrak{m}}\to L:\quad x\mapsto c_x\quad\text{for }x\in\pmb{x},$$
where $\mathfrak{m}$ is the ideal $(x-c_x\mid x\in\pmb{x})$, restricts to a field homomorphism $M\to L$ over $F$.
\end{prop}
\begin{proof}
If $\trdeg_F L = 0$, both $L/F$ and $M/F$ are algebraic extensions, so $M$ lies in the algebraic closure $L$ of the extension $L(\pmb{x})/L$ and we may take $h(\pmb{x})=1$. The case for $\trdeg_F L>0$ follows from the proof of \cite[Lemma 1]{roquette1964isomorphisms}.
\end{proof}

Here is a differential analogue of Proposition \ref{prop:roquette}.

\begin{prop}\label{prop:fg-embed}
Let $L/F$ be a finitely generated differential field extension, and let $M/F$ be a differential field subextension of $L\anglebrac{t}/F$ for some differentially transcendental element $t$ over $L$. Assume that $\dtrdeg_FM\le \dtrdeg_FL$, and that $F$ has uncountable differential transcendence degree over some differential subfield. Then there exists an element $c\in F$ such that the
differential $L$-algebra homomorphism 
$$L\{ t \}_{\mathfrak{m}}\to L:\quad t\mapsto c,$$
where $\mathfrak{m}$ is the differential ideal $(t^{(i)}-c^{(i)}\mid i\ge 0)$, restricts to a differential field homomorphism of $M\to L$ over $F$.
\end{prop}
\begin{proof}

Since $L/F$ is a finitely generated differential field extension, the subextension $M/F$ of $L\anglebrac{ t }/ F$ is also a finitely generated differential field extension by Proposition \ref{prop:fg}. Let $\pmb{y}$ and $\pmb{z}$ be finite sets that satisfy $M = F\anglebrac{\pmb{y}}$ and $L = F\anglebrac{\pmb{z}}$. The subfields
$$
M_i = F(\pmb{y},\pmb{y}^{(1)},...,,\pmb{y}^{(i)}),\qquad L_i = F(\pmb{z},\pmb{z}^{(1)},...,,\pmb{z}^{(i)})
$$
have finite transcendence degree over $F$.

Fix $i\ge 0$. Since $M\subseteq L\anglebrac{t}$, $M_i$ is a subfield of $L_j\anglebrac{ t } = L_j({ t },{ t }^{(1)},{ t }^{(2)},...)$ for some least $j\ge 0$. By Proposition \ref{prop:roquette}, there exists a nonzero polynomial $h_{i}$ over $L_j$ in the algebraic indeterminates $t,t^{(1)}, t ^{(2)},...$ satisfying the following: for any $c\in L_j$ such that $h_{i}(c,c^{(1)},c^{(2)},...)\neq 0$, the $L_j$-algebra homomorphism 
$$L_j\{ t \}_{\mathfrak{n}}\to L_j:\quad t^{(i)}\mapsto c^{(i)}\quad\text{for all } i\ge 0,$$
where $\mathfrak{n}$ is the ideal $(t^{(i)}-c^{(i)}\mid i\ge 0)$ in $L_{j}\{t\}$,
restricts to an $F$-algebra homomorphism $M_{i}\to L_j$. Composing $M_{i}\to L_j$ by the inclusion $L_{j}\subseteq L$, the map $M_{i}\to L$ is then the restriction of $L\{{ t }\}_{\mathfrak{m}}\to L$, where $\mathfrak{m}$ is the differential ideal $(t^{(i)}-c^{(i)}\mid i\ge 0)$ in $L\{t\}$. We will choose $c$ in $F$ instead of the bigger fields $L_j$.

Since $\bigcup M_{i} = M$, any $c\in F$ that satisfies $h_{i}(c,c^{(1)},c^{(2)},...) \neq 0$ for all $i\ge 0$ will induce a differential field homomorphism $M\to L$ over $F$. Let $S$ be the set of coefficients in $L$ of the differential polynomials $h_i$ for all $i\ge 0$. Since $S$ is countable and $\dtrdeg_{F_0}F$ is uncountable for some differential subfield $F_0$ of $F$, there are uncountably many elements $c$ in $F$ that are differentially transcendental over $F_0\anglebrac{S}$. Each such $c\in F$ satisfies the inequalities $h_{i}(c,c^{(1)},c^{(2)},...) \neq 0$ for all $i\ge 0$, and in turn yields a desired map $M\to L$.
\end{proof}

We now prove an analogue of \cite[Proposition 8.2.5]{jensen2002generic} and \cite[Corollary 8.2.6]{jensen2002generic}:

\begin{prop}\label{prop:ded-rational-ext}
Suppose that $F$ has uncountable differential transcendence degree over some differential subfield. Then given a $G$-Picard-Vessiot extension $L/K$ and differentially transcendental elements $t_1,...,t_n$ over $L$, we have $\ded_{F}(L\anglebrac{t_1,...,t_n}/K\anglebrac{t_{1},...,t_{n}}) = \ded_F(L/K)$. Here, $G$ acts on $L\anglebrac{ t_1,...,t_n}$ by the action of $G$ on $L$ and the trivial action on the $t_i$.
\end{prop}
\begin{proof}
It suffices to prove the case $n=1$ since the general case is a repeated application of this case. Let $t=t_1$. Clearly $L\anglebrac{t}/K\anglebrac{t}$ descends to $L/K$, so $\ded_{F}(L\anglebrac{t}/K\anglebrac{t}) \le \ded_F(L/K)$. Now suppose that $L\anglebrac{t}/K\anglebrac{t}$ descends to a $G$-Picard-Vessiot extension $L_0/K_0$ with $\dtrdeg_{F}K_0\le \dtrdeg_{F}K$. We may assume that the inclusion $L_0 \subseteq L\anglebrac{ t}$ is $G$-equivariant. By Proposition \ref{prop:fg-embed}, there exists $c\in F$ such that the differential homomorphism
$$\varphi:L\{ t \}_{\mathfrak{m}}\to L:\quad t\mapsto c,$$
where $\mathfrak{m}$ is the differential ideal $(t^{(i)}-c^{(i)}\mid i\ge 0)$, restricts to a differential field homomorphism of $L_0$ into $L$ over $F$. Since $c$ is in $F$, the map $\varphi$ is $G$-equivariant. By Lemma \ref{lem:descent-tfae}, $L/K$ descends to $L_0/K_0$.
\end{proof}

\begin{prop}\label{prop:ded-bounds}
Let $L/K$ be the $\GL_n$-Picard-Vessiot extension for the general $n\times n$ matrix differential equation $\delta(Y)=AY$. Let $(\mathbb{Z}/k\mathbb{Z})^n$ for $k\ge 2$ embed diagonally into $\GL_n$ as the group of diagonal matrices whose orders divide $k$, and let $\mathbb{G}_a^{n-1}$ embed into $\GL_n$ via the map 
$\pmb{a}\mapsto \begin{pmatrix}
1 & \pmb{a}\\
0 & I_{n-1}\\
\end{pmatrix}$, where $\pmb{a} = (a_1,...,a_{n-1})$.
Let $H$ be a closed subgroup of $\GL_n$.
\begin{enumerate}[label={(\arabic*)}]
    \item\label{emb:1}  If $H$ contains the above embedding of $(\mathbb{Z}/k\mathbb{Z})^n$, then $\ded_F(L/L^H) = n$.
    \item\label{emb:2}  If $H$ contains the above embedding of $\mathbb{G}_a^{n-1}$, then $n\ge \ded_F(L/L^H) \ge n-1$.
\end{enumerate}
\end{prop}
\begin{proof}
Let $H'\le H\le \GL_{n}$. The upper bound $n\ge \ded_{F}(L/L^{H})$ in \ref{emb:1} and \ref{emb:2} follows from Proposition \ref{prop:gamma=GL} and Corollary \ref{cor:descent-subPV}\ref{descent-ded-1} since we have the inequalities
$$n\ge \gamma(n) \ge \ded_F(L/L^{\GL_n})\ge \ded_F(L/L^H)\ge \ded_F(L/L^{H'}).$$ 

Let $\pmb{t}$ be an uncountable set of differential indeterminates over $L$. Since
$\ded_{F}(L/K)\ge \ded_{F\anglebrac{ \pmb{t}}}(L\anglebrac{ \pmb{t}}/K\anglebrac{ \pmb{t}})$, we can replace $L$, $K$, and $F$ by $L\anglebrac{ \pmb{t}}$, $K\anglebrac{ \pmb{t}}$, and $F\anglebrac{ \pmb{t}}$ to further assume that $F$ has uncountable differential transcendence degree over some differential subfield.

\ref{emb:1}. We now prove the inequality $\ded_F(L/L^{H'}) \ge n$ for $H' = (\mathbb{Z}/k\mathbb{Z})^n$. As a differential field, $L$ is generated by the $n^2$ entries of a fundamental solution matrix $(y_{ij})$ of $\delta(Y)=AY$  over $K$. The set $S$ consisting of elements 
\begin{enumerate}[label={(\alph*)}]
    \item\label{part-1a} $y_{1j}^{k}$ for $j = 1,...,n$, and
    \item\label{part-1b} $y_{1j}^{k-1}y_{ij}$ for $i=2,...,n$ and $j=1,...,n$
\end{enumerate}
is $H'$-invariant. Thus $F\anglebrac{ S}\subseteq L^{H'}$ and $[L:F\anglebrac{S}]\ge k^n$. Since $L$ is obtained by adjoining the set
$$T \coloneqq \left\lbrace y_{11},y_{12},...,y_{1n}\right\rbrace$$
to $F\anglebrac{ S}$ and since $y_{1j}^k\in F\anglebrac{S}$, we have $[L:F\anglebrac{S}]\le k^n$. Thus $L^{H'} = F\anglebrac{ S}$ by degree considerations.

From this explicit description, $L/L^{H'}$ descends to $L_0/L_0^{H'}$ where $L_0 = F\anglebrac{ T}$. The extension $L_0/L_0^{H'}$ is isomorphic to the one constructed in Corollary \ref{cor:ded-gm-ga} so $\ded_{F}(L_0/L_0^{H'}) = n$. Moreover $L$ is obtained by adjoining the differentially algebraically independent elements of the form \ref{part-1b}
to $L_0$. By Proposition \ref{prop:ded-rational-ext},
\begin{equation*}
    \ded_{F}(L/L^{H'}) \ge \ded_{F}(L_0/L_0^{H'}) = n.
\end{equation*}

\ref{emb:2}. We similarly prove $\ded_{F}(L/L^{H'})\ge n-1$ for $H' = \mathbb G_{a}^{n-1}$. The set $S$ consisting of elements
\begin{enumerate}[label={(\alph*\textquotesingle)}]
    \item\label{part-1} $ \delta(y_{i2}/y_{i1})$ for $i=2,...,n$,
    \item\label{part-2} $y_{11}, y_{21},...,y_{n1}$,
    \item\label{part-3}$y_{i2} - y_{ik}$ for $i=1,...,n$ and $k= 3,...,n$, and
    \item\label{part-4} $y_{11}y_{22}-y_{12}y_{21}$,
\end{enumerate}
is $H'$-invariant. Thus $F\anglebrac{S}\subseteq L^{H'}$ and $\trdeg_{F\anglebrac{S}} L\ge n-1$. Since $L$ is obtained by adjoining the set
$$T \coloneqq \left\lbrace y_{22}/y_{21},y_{32}/y_{31},...,y_{n2}/y_{n1}\right\rbrace$$
to $F\anglebrac{ S}$ and since $\delta(T)\subseteq F\anglebrac{S}$, the extension $L/F\anglebrac{S}$ is a $\mathbb G_{a}^{m}$-Picard-Vessiot extension for some $m\le n-1$. Thus $\trdeg_{F\anglebrac{S}}L= n-1$, so $L^{H'}/F\anglebrac{S}$ is a finite extension. Since $\mathbb G_{a}^{m}$ is a connected algebraic group, we have $L^{H'} = F\anglebrac{S}$.

Next we define $L_0 = F\anglebrac{T}$ and note that $L/L^{H'}$ descends to the extension $L_0/L_{0}^{H'}$ which is isomorphic to the one constructed in Corollary \ref{cor:ded-gm-ga} so $\ded_{F}(L_0/L_0^{H'}) = n-1$. Moreover $L$ is obtained by adjoining the differentially algebraically independent elements of the form \ref{part-2} -- \ref{part-4} to $L_0$. By Proposition \ref{prop:ded-rational-ext},
\begin{equation*}
    \ded_{F}(L/L^{H'}) \ge \ded_{F}(L_0/L_0^{H'}) = n-1.\qedhere
\end{equation*}
\end{proof}

Proposition \ref{prop:ded-bounds} applied to $H=\GL_n$, combined with Proposition \ref{prop:gamma=GL}, gives

\begin{thm}\label{thm:gn=n}
For all $n\ge 1$, $\gamma(n)=n$.
\end{thm}

\section{Generic Picard-Vessiot extensions}\label{section:generic-pv}

In this section, we prove lower bounds on the number of parameters in a generic Picard-Vessiot extension, following the method of \cite[Section 7]{buhler1997essential}. We again assume the field of constants $C$ of $F$ is algebraically closed, and all differential fields in discussion contain $F$ and have field of constants $C$.

\begin{defi}
Let $K = F\anglebrac{ y_1,...,y_m}$ with $y_1,...,y_m$ differential indeterminates over $F$. A $G$-Picard-Vessiot extension $L/K$ is \emph{generic} if for every $G$-Picard-Vessiot extension $L'/K'$, the following two conditions hold:
\begin{enumerate}
    \item there exists some matrix differential equation $\delta(Y)=A(y_1,...,y_m)Y$ over $K$ whose Picard-Vessiot extension is $L/K$;
    \item there exist elements $a_{1},...,a_m\in K'$ such that $A(a_1,...,a_m)$ is defined and $L'/K'$ is the Picard-Vessiot extension for the differential equation $\delta(Y)=A(a_1,...,a_m)Y$.
\end{enumerate}
\end{defi}

The above definition of \emph{generic} is slightly more inclusive than that given in \cite{juan2007picard} and \cite[Section 6]{juan2016generic}. In particular, the following result applies to the generic extensions constructed there.

\begin{prop}\label{prop:generic-PV}
Let $L/K$ and $L'/K'$ be $G$-Picard-Vessiot extensions. If $L/K$ is generic, $\dtrdeg_{F}K\ge \ded_{F}(L'/K')$.
\end{prop}
\begin{proof}
Since $L/K$ is generic, there exists an equation $\delta(Y)=A(y_1,...,y_m)Y$ whose Picard-Vessiot extension is $L/K$, and there exist $a_1,...,a_m\in K'$ such that $\delta(Y)=A(a_1,...,a_m)Y$ has Picard-Vessiot extension $L'/K'$. Since $\delta(Y)=A(a_1,...,a_m)Y$ has coefficients in $K''=F\anglebrac{ a_1,...,a_m}$, $L'/K'$ is defined over $K''$. Thus 
\begin{equation*}
    \dtrdeg_{F}K\ge \dtrdeg_{F}K''\ge \ded_{F}(L'/K').\qedhere
\end{equation*}
\end{proof}

Proposition \ref{prop:ded-bounds} and Proposition \ref{prop:generic-PV} together give

\begin{cor}\label{cor:generic-PV}
Let $L/K$ be a generic $G$-Picard-Vessiot extension. Then
\begin{enumerate}
    \item $\dtrdeg_F K\ge n$ if $G$ is either $\GL_n$, $\mathbb{G}^n_m$, $\mathbb{G}^n_a$, $(\mathbb{Z}/k\mathbb{Z})^n$ (for $k\ge 2$), or $\OO_n$, and
    \item $\dtrdeg_F K\ge n-1$ if $G$ is either $\SL_n$, $\mathbb{U}_n$, or $\SO_n$.
\end{enumerate}
\end{cor}
Here $\mathbb U_{n}$ is the group of upper triangular matrices with ones on the diagonal.

The differential essential dimension for some other groups like $\PGL_n$ are bounded in a forthcoming paper.

\bibliographystyle{amsplain}

\end{document}